\documentclass{article}
\title{The Herbrand Topos}
\author{Benno van den Berg}
\date{April 18, 2013}
\usepackage{ast2}
\newcommand{\N}{\mathbb{N}}
\newcommand{\R}{\mathbb{R}}

\DeclareMathOperator{\st}{st} 

\newcommand{\existsst}{\exists^{\st{}}\!}

\begin{document}

\maketitle

\begin{abstract}
\noindent We define a new topos, the \emph{Herbrand topos}, inspired by the modified realizability topos and our earlier work on Herbrand realizability. We also introduce the category of Herbrand assemblies and characterise these as the $\lnot\lnot$-separated objects in the Herbrand topos. In addition, we show that the category of sets is included as the category of $\lnot\lnot$-sheaves and prove that the inclusion functor preserves and reflects validity of first-order formulas.
\end{abstract}

\section{Introduction}

In \cite{bergbriseidsafarik12} the author, together with Eyvind Briseid and Pavol Safarik, hit upon a new realizability interpretation in an attempt to find computational content in arguments performed in nonstandard analysis. This new interpretation, which was a variant of modified realizability, was dubbed Herbrand realizability. Our investigations in  \cite{bergbriseidsafarik12} were entirely proof-theoretic; the question was whether it would be possible to understand Herbrand realizability from a semantic point of view as well. This paper shows that that is indeed the case.

To develop this semantics we use topos theory (for which see \cite{maclanemoerdijk92, johnstone02a, johnstone02b}). This choice was motivated by the fact that the notion of a topos is the most comprehensive notion of model for a constructive system we have available, incorporating topological, sheaf and Kripke models, as well as various realizability and functional interpretations. In addition, it shows that these interpretations can be made to work for full higher-order arithmetic. The starting point for this paper was the theory of realizability toposes (beginning with \cite{hyland82} and surveyed in \cite{vanoosten08}): indeed, the topos most closely related to the topos we will introduce here is the modified realizability topos (for which, see \cite{vanoosten97, vanoosten08}). 

In order to arrive at the modified realizability topos, one has to abstract considerably from Kreisel's original definition \cite{kreisel59}. First of all, one fixes the hereditarily effective operations (HEO) as a model of G\"odel's $T$. Then a type gets identified with a certain inhabited set of codes and a set of realizers of that type will simply be subset of that set. The step that Grayson took in \cite{grayson81} was to take as truth values any pair $(A_0, A_1)$ where $A_0$ and $A_1$ are two sets of codes (often called the actual realizers and the potential realizers, respectively) with $A_0 \subseteq A_1$ and $A_1$ containing a fixed element. One can build a tripos around such pairs and in the associated topos the finite types will be interpreted as the hereditarily effective operations.

In order to define the Herbrand topos, we make a similar move. Recall from \cite{bergbriseidsafarik12} that in order to realize
\[ (\existsst n \in \NN) \, \varphi(n) \]
it suffices to supply a finite list of natural numbers $(n_1, \ldots, n_k)$ such that $\varphi(n_i)$ is realized for some $i \leq k$. Abstracting away from the details, this means that potential realizers are finite lists of natural numbers, while the actual realizers are those finite lists $(n_1, \ldots, n_k)$ which contain an $n_i$ which works (this is similar to the idea of Herbrand disjunctions in proof theory; hence the name). Abstracting even further, we say that truth values in the Herbrand topos are pairs of sets of codes $(A_0, A_1)$ such that $A_0$ consists of finite sequences all whose elements belong to $A_1$ and which is closed upwards (by this we mean that it is closed under supersets, if we regard finite sequences as representatives for their set of components; see also Lemma 5.4 in \cite{bergbriseidsafarik12}). This paper shows that on the basis of these pairs one can construct a tripos, whose associated topos we will call the Herbrand topos. 

The Herbrand topos turns out to have several features in common with other realizability toposes. It has an interesting subcategory consisting of the $\lnot\lnot$-separated objects (we will call these the Herbrand assemblies) and the category of sets is included as a subtopos via the $\lnot\lnot$-topology. What is very unusual, however, is that this inclusion functor, which we will call $\nabla$, preserves and reflects the validity of first-order logic; in fact, $\nabla$ preserves and reflects the structure of a locally cartesian closed pretopos. In particular, $\nabla 2 = 2$ in the Herbrand topos.

This is a striking illustration of the fact that in the Herbrand topos disjunction has essentially no constructive content. Indeed, in order for a disjunction $\varphi \lor \psi$ to be realized it is sufficient that one of the two disjuncts is realized; but a realizer for $\varphi \lor \psi$ need not say which disjunct it is that is actually realized. This fact explains many of the features of the Herbrand topos: why it believes in the law of excluded middle for $\Pi^0_1$-formulas, and why it does not believe in Church's thesis or in continuity principles.

However, arithmetic in the Herbrand topos is not classical. This is due to the fact that existential quantifiers still have some constructive content. Admittedly, this content is less than is usually the case, but it is still strong enough to rule out Markov's principle.

Finally, there are two other properties of the Herbrand topos which are worth mentioning. First of all, because $\nabla$ preserves and reflects first-order logic, $\nabla A$ will be a nonstandard model of arithmetic in the Herbrand topos for any nonstandard model $A$ (actually, it will also be a nonstandard model when $A$ is the standard model). This is interesting, because realizability toposes are unfavourable terrain for nonstandard models of arithmetic (see \cite{mccarty87}).

A proof-theoretic feature of the Herbrand topos which is worth stressing is that in it the Fan Theorem holds (see \refprop{fan} below); for this it is not necessary to assume its validity in the metatheory. The proof should look very familiar to anyone who is aware of the bounded modified realizability interpretation and its properties (for which, see \cite{ferreiranunes06}).

The contents of the paper are as follows. In Section 2 we will explain the notation that we will use, in so far as it is not standard. We define the Herbrand tripos and topos in Section 3. Section 4 introduces the Herbrand assemblies and proves that they form a locally cartesian closed regular category with finite stable colimits and a natural numbers object. Then, in Section 5, we characterise these Herbrand assemblies as the $\lnot\lnot$-separated objects in the Herbrand topos; in addition, we show that the category of sets is included as the $\lnot\lnot$-sheaves and that the inclusion functor preserves and reflects the structure of a locally cartesian closed pretopos. In Section 6, we characterise the projectives in the Herbrand assemblies and show that the natural numbers object is projective in the category of Herbrand assemblies, but not in the Herbrand topos. Section 7 studies the logical properties of the Herbrand topos and, finally, in Section 8 we discuss some further developments.

In this paper we assume familiarity with the theory of triposes and partial combinatroy algebras; for the necessary background information, we recommend \cite {vanoosten08}.

We would like to thank Jaap van Oosten, Wouter Stekelenburg and the referee for useful comments. The author was supported by a grant from the Netherlands Organisation for Scientific Research (NWO).

\section{Notation}

Throughout this paper, $\pca{P}$ will be a fixed nontrivial partial combinatory algebra (pca). It is well known that we can code finite sequences of elements in $\pca{P}$ as elements of $\pca{P}$. If $n$ is such a code, then we write $|n|$ for the length of the sequence it codes and $n_i$ for the $i$th projection (where $n_1$ is the first element of the sequence and $n_{|n|}$ the last). We will write $<n_1, \ldots, n_k>$ for the code of the sequence $(n_1, \ldots, n_k)$ and $<>$ for the code of the empty sequence. Moreover, $*$ will denote the operation of concatenation of coded sequences and $a_1 * a_2 * \ldots * a_k$ will stand for the result of concatenating $a_1$ till $a_k$. If $k = 0$ (i.e., if we are taking the empty concatenation), then the result should be the code of the empty sequence $<>$.

If $A \subseteq \pca{P}$ is some subset of $\pca{P}$, then we will write $!A$ for the set of codes of finite sequences all whose elements belong to $A$ (note that this will always include the empty sequence). For our purposes it will be important that $!A$ carries a preorder structure with $m \preceq n$ if
\[ (\forall i \leq |m|) \, (\exists j \leq |n|) \, m_i = n_j. \]

We will write ${\bf p}$ for the pairing operator and denote the $n$th Church numeral simply by $n$. In addition, if $A$ and $B$ are subsets of $\pca{P}$, we will write
\begin{eqnarray*}
A \& B & = &  \{ {\bf p}0a \, : \, a \in A \} \cup \{ {\bf p}1b \, : \, b \in B \}, \\
A \otimes B & = & \{ {\bf p}ab \, : \, a \in A, b \in B \}.
\end{eqnarray*}
Observe that there is an exponential isomorphism
\[ !(A \& B) \cong !A \otimes !B, \]
where both the map itself and and its inverse are order-preserving and represented by elements in the pca, independent of the specific $A$ and $B$. We will often implicitly use this isomorphism and regard elements of $!(A \& B)$ as pairs ${\bf p}xy$ with $x \in !A$ and $y \in !B$.

\section{The Herbrand tripos}

The Herbrand topos will be obtained from the \emph{Herbrand tripos}. To construct this tripos first put
\[ \Sigma = \{ (A_0, A_1) \, : \, A_0, A_1 \subseteq \pca{P}, A_0 \subseteq !A_1 \mbox{ and } A_0 \mbox{ is upwards closed in } !A_1 \}, \]
where $A_0$ being upwards closed in $!A_1$ means that $m \in A_0, n \in !A_1$ and $m \preceq n$ imply $n \in A_0$. For $X \in \Sigma$, we will denote the result of projecting on the first and second coordinate by $X_0$ and $X_1$. As for modified realizability, we will refer to the elements of $X_0$ as the \emph{actual realizers} and the elements of $!X_1$ as the \emph{potential realizers}.

If $X$ is any set, then we preorder $\Sigma^X$ as follows: $(\phi: X \to \Sigma) \leq (\psi: X \to \Sigma)$ if there is an element $r \in \pca{P}$ such that for all $x \in X$ and $n \in \pca{P}$
\[ \mbox{if } n \in !\phi(x)_1 \mbox{, then } r \cdot n \downarrow \mbox{ and } r \cdot n \in !\psi(x)_1 \]
and
\[ \mbox{if } n \in \phi(x)_0 \mbox{, then } r \cdot n \in \psi(x)_0. \]
Finally, if $f: X \to Y$ is any function, then reindexing $f^*: \Sigma^Y \to \Sigma^X$ is given simply by precomposition.

\begin{theo}{Herbrandtripos}
The indexed preorder $\Sigma^X$ defined above is a tripos.
\end{theo}

The associated topos we will call the Herbrand topos and denote by $\HT[\pca{P}]$.

\begin{proof}
We first verify that $\Sigma$ has the structure of a Heyting prealgebra. The top and bottom element are $(!\pca{P}, \pca{P})$ and $(\emptyset, \emptyset)$, respectively. The conjunction $(C_0, C_1) = (A_0, A_1) \land (B_0, B_1)$ is given by
\begin{eqnarray*}
C_1 & = & A_1 \& B_1, \\
C_0 & = & \{ {\bf p}ab \, : \, a \in A_0 \mbox{ and } b \in B_0 \}
\end{eqnarray*}
(making use of the exponential isomorphism). The disjunction $(C_0, C_1) = (A_0, A_1) \lor (B_0, B_1)$ is given by
\begin{eqnarray*}
C_1 & = & A_1 \& B_1, \\
C_0 & = & \{ {\bf p}ab \, : \, a \in A_0 \mbox{ or } b \in B_0 \}
\end{eqnarray*}
(again making use of the exponential isomorphism). To see that this works, suppose that $(A_0, A_1) \leq (D_0, D_1)$ is tracked by $r$ and $(B_0, B_1) \leq (D_0, D_1)$ is tracked by $s$. Then $(C_0, C_1) \leq (D_0, D_1)$ is tracked by the following map $t$: $t({\bf p}ab) = ra * sb$. This works, because for ${\bf p}ab \in C_0$ we have $a \in A_0$ or $b \in B_0$: in the former case, we have $ra \in D_0$, in the latter $sb \in D_0$. In either case we have $ra * sb$ because $D_0$ is upwards closed.

The implication $(C_0, C_1) = (A_0, A_1) \to (B_0, B_1)$ is given by
\begin{eqnarray*}
C_1 & = & \{ c \in \pca{P} \, : \, (\forall m \in !A_1) \, c \cdot m \downarrow \mbox{ and } c \cdot m \in !B_1 \}, \\
C_0 & = & \uparrow_{!C_1} \{ <c> \, : \, c \in C_1 \mbox{ and } (\forall m \in A_0) \, c \cdot m \in B_0 \}.
\end{eqnarray*}
To see that this works, note that the evaluation map $(C_0, C_1) \land (A_0, A_1) \to (B_0, B_1)$ is tracked by the function which maps ${\bf p}(<c_1,\ldots, c_k >, a)$ to
\[ c_1(a) * \ldots * c_k(a). \]
It is clear that this Heyting prealgebra structure lifts to each $\Sigma^X$.

As said, the reindexing functors are given by precomposition. We now check that these have both adjoints satisfying the Beck-Chevalley condition. So suppose we have $\phi: X \to \Sigma, \chi: Y \to \Sigma$ and $f: X \to Y$. The existential quantifier can be defined as follows:
\begin{eqnarray*}
\exists_f(\phi)(y)_1 & = & \bigcup_{x \in f^{-1}(y)} !\phi(x)_1, \\
\exists_f(\phi)(y)_0 & = & \uparrow_{!\exists_f(\phi)(y)_1} \{ <n> \, : \, (\exists x \in f^{-1}(y)) \, n \in \phi(x)_0 \}.
\end{eqnarray*}
To see this, note that 
\begin{enumerate}
\item we have $\exists_f(\phi) \leq \chi$ if there is an $r \in \pca{P}$ sending for each $y \in Y$ elements from $!\exists_f(\phi)(y)_1$ to elements in $!\chi(y)_1$, in such a way that if $m \in \exists_f(\phi)(y)_0$, then $r \cdot m \in \chi(y)_0$. 
\item And that we have $\phi \leq f^*(\chi)$, if there is an $s \in \pca{P}$ sending for each $x \in X$ elements in $!\phi(x)_1$ to elements in $!\chi(fx)_1$, in such a way that if $m \in \phi(x)_0$, then $s \cdot m \in \chi(fx)_0$.
\end{enumerate}
To construct such an $s$ from an $r$, put
\[ s(m) = r(<m>) \]
and to construct such an $r$ from an $s$, put
\[ r(<m_1, \ldots, m_k>) = s(m_1) * \ldots s(m_k). \]

The universal quantifier is constructed as follows:
\begin{eqnarray*}
\forall_f(\phi)(y)_1 & = & \{ a \in \pca{P} \, : \, (\forall x \in f^{-1}(y)) \, (\forall b \in \pca{P}) \, a \cdot b \downarrow \mbox{ and } a \cdot b \in  !\phi(x)_1 \}, \\
\forall_f(\phi)(y)_0 & = & \uparrow_{!\forall_f(\phi)(y)_1} \{ <a> \, : \, a \in \forall_f(\phi)(y)_1 \mbox{ and } \\
& &  (\forall x \in f^{-1}(y)) \, (\forall b \in \pca{P}) \, a \cdot b \in \phi(x)_0 \}.
\end{eqnarray*}
To see this, note that 
\begin{enumerate}
\item we have $f^*(\chi) \leq \phi$ if there is an $r \in \pca{P}$ sending for each $x \in X$ elements from $!\chi(fx)_1$ to elements in $!\phi(x)_1$, in such a way that if $m \in \chi(fx)_0$, then $r \cdot m \in \phi(x)_0$. 
\item And that we have $\chi \leq \forall_f(\phi)$, if there is an $s \in \pca{P}$ sending for each $y \in Y$ elements in $!\chi(y)_1$ to elements in $!\forall_f(\phi)(y)_1$, in such a way that if $m \in \chi(y)_0$, then $s \cdot m \in \forall_f(\phi)(y)_0$.
\end{enumerate}
To construct such an $s$ from an $r$, put
\[ s = \lambda p.<\lambda q.r \cdot p> \]
and to construct such an $r$ from an $s$, put
\[ r = \lambda p.\big[ \, (s \cdot p)_0(k) * \ldots *(s \cdot p)_{|s \cdot p|}(k)\,  \big]. \]
It is easy to see that the universal quantifier, and hence also the the existential quantifier, satisfies Beck-Chevalley.

Finally we need to construct a generic element; but that can quite straightforwardly be taken to be the identity map $\id: \Sigma \to \Sigma$.
\end{proof}

\section{Herbrand assemblies}

In this section we introduce the Herbrand assemblies and prove that they form a locally cartesian closed regular category with stable colimits. Later we will characterise them as the $\lnot\lnot$-separated objects in the Herbrand topos. Of course, it follows from this that the category is a regular locally cartesian closed category with stable colimits, but we need an explicit description of this structure later; in addition, such a description is, we believe, of independent interest.

\begin{defi}{herbrassemblies}
A \emph{Herbrand assembly} over $\pca{P}$ is a triple $(A, {\cal A}, \alpha)$ in which
\begin{itemize}
\item $A$ is a set, 
\item ${\cal A}$ is a subset of $\pca{P}$, and
\item $\alpha: A \to {\rm Pow}^{upcl}_i(!{\cal A})$ is a function whose codomain ${\rm Pow}^{upcl}_i(!{\cal A})$ consists of the subsets $X$ of $!{\cal A}$ that are inhabited and upwards closed in $!\pca{A}$.
\end{itemize}
A morphism $f: (B, {\cal B}, \beta) \to (A, {\cal A}, \alpha)$ of Herbrand assemblies is a function $f: B \to A$ for which there is an $n \in \pca{P}$ such that
\begin{itemize}
\item for all $m \in !{\cal B}$, the expression $n \cdot m$ is defined and its value belongs to $!{\cal A}$, 
\item and if $b \in B$ and $m \in \beta(b)$, then $n \cdot m \in \alpha(fb)$.
\end{itemize}
We will say that such an $n$ \emph{tracks $f$} or is a \emph{tracking of $f$}. This clearly defines a category: we will denote it by $\HAsm[\pca{P}]$.
\end{defi}

\begin{lemm}{HAsmcartesian}
The category $\HAsm[\pca{P}]$ has finite limits.
\end{lemm}
\begin{proof}
The terminal object is $(1, {\cal C}, \gamma)$ with ${\cal C} = \pca{P}$ and $\gamma(*) = !\pca{P}$.

Equalizers as in 
\diag{ (C, {\cal C}, \gamma) \ar@{ >->}[r] & (B, {\cal B}, \beta) \ar@<+1ex>[r]^f \ar@<-1ex>[r]_g & (A, {\cal A}, \alpha) }
can be computed by putting $C = \{ b \in B \, : \, fb = gb \}$, ${\cal C} = {\cal B}$ and $\gamma = \beta \upharpoonright C$. 

A product $(C, {\cal C}, \gamma) = (A, {\cal A}, \alpha) \times (B, {\cal B}, \beta)$ can be constructed by putting
\begin{eqnarray*}
C & = & A \times B, \\
{\cal C} & = & {\cal A} \& {\cal B}, \\
\gamma(a, b) & = & \{ \, {\bf p}ab \, : \, a \in \alpha(a), b \in \beta(b) \, \}.
\end{eqnarray*}
\end{proof}

\begin{lemm}{HAsmregular}
The category $\HAsm[\pca{P}]$ is regular.
\end{lemm}
\begin{proof}
We claim that a map $f: (B, {\cal B}, \beta) \to (A, {\cal A}, \alpha)$ is monic precisely when the underlying function $f$ is injective (this should be clear), and is a cover precisely when there is an element $r \in \pca{P}$ such that
\begin{itemize}
\item for all $n \in !{\cal A}$ the expression $r \cdot n$ is defined and belongs to $!{\cal B}$, and
\item there is for any $a \in A$ and $n \in \alpha(a)$  an element $b \in f^{-1}(a)$ with $r \cdot n \in \beta(b)$.
\end{itemize}
Let us call maps which have these two properties \emph{super epis}. It is not hard to check that super epis are covers and that they are stable under pullback, so the proof will be finished once we show that every map can be factored as a super epi followed by a mono. But if $f: (B, {\cal B}, \beta) \to (A, {\cal A}, \alpha)$ is any map, then it factors through $(C, {\cal C}, \gamma)$ with $C = {\rm Im}(f)$, ${\cal C} = {\cal B}$ and 
$\gamma(a) = \bigcup_{b \in f^{-1}(a)} \beta(b)$. Moreover, the obvious maps $(B, {\cal B}, \beta) \to (C, {\cal C}, \gamma)$ and $(C, {\cal C}, \gamma) \to (A, {\cal A}, \alpha)$ are a super epi and a mono, respectively.
\end{proof}

\begin{lemm}{HAsmlextensive}
The category $\HAsm[\pca{P}]$ has stable sums and coequalizers.
\end{lemm}
\begin{proof}
The initial object is $(0, \pca{P}, \gamma)$ with $\gamma$ the empty map.

The sum $(C, {\cal C}, \gamma) = (A, {\cal A}, \alpha) + (B, {\cal B}, \beta)$ is given by $C = A + B$, ${\cal C} = {\cal A} \& {\cal B}$, with $\gamma(a) = \{ {\bf p}xy \, : \, x \in \alpha(a), y \in !\pca{B} \} $ and $\gamma(b) = \{ {\bf p}xy \, : \, x \in !\pca{A}, y \in \beta(b) \}$.

A coequalizer
\diag{ (B, {\cal B}, \beta) \ar@<+1ex>[r]^f \ar@<-1ex>[r]_g & (A, {\cal A}, \alpha) \ar@{->>}[r]^q & (C, {\cal C}, \gamma)}
is computed by letting $C$ be the coequalizer in the category of sets, ${\cal C} = {\cal A}$ and $\gamma(c) = \bigcup_{a \in q^{-1}(c)} \alpha(a)$.
\end{proof}

\begin{lemm}{HAsmlccc}
The category $\HAsm[\pca{P}]$ is locally cartesian closed.
\end{lemm}
\begin{proof} Assume $f: (B, {\cal B}, \beta) \to (A, {\cal A}, \alpha)$ and $g: (S, {\cal S}, \sigma) \to (B, {\cal B}, \beta)$ are two morphisms of Herbrand assemblies. The object $\prod_f(g) = (T, {\cal T}, \tau)$ is computed as follows:
\begin{eqnarray*}
T & = & \{ (a \in A, t: B_a \to S) \, : \, gt = \id_{B_a} \mbox{ and } t \mbox{ is tracked} \, \}, \\
{\cal T} & = & {\cal A} \, \&  \, \{  n \in \pca{P} \, : \, (\forall m \in !{\cal B}) \, n \cdot m \downarrow \mbox{ and } n \cdot m \in !{\cal S} \, \}, \\
\tau(a, t) & = & \{  \, {\bf p}(m, <n_1, \ldots, n_k>) \, : \, m \in \alpha(a) \mbox{ and }  \\
& & \mbox{there is an } n_i \mbox{ tracking } t \, \}.
\end{eqnarray*}
The evaluation map $f^*\prod_f(g) \to g: ((a, t), b) \mapsto t(b)$ is tracked by a code for the function $L$ defined by
\[ L({\bf p}(<n_1, \ldots, n_k>, m)) = n_1(m) * \ldots * n_k(m). \]
\end{proof}

\begin{lemm}{HAsmnno}
The category $\HAsm[\pca{P}]$ has a natural numbers object (nno).
\end{lemm}
\begin{proof}
The nno is given by $(\N, {\cal N}, \nu)$ with ${\cal N}$ the collection of Church numerals in $\pca{P}$ and $\nu(n) = \uparrow_{!\cal N} <n>$. For suppose a structure $1 \to (A, {\cal A}, \alpha) \to (A, {\cal A}, \alpha)$ is given and the first map is tracked by $p$ and the second by $q$. Define $s$ by recursion as $s0 = p0$ and $s(n+1) = q(s(n))$, and $t$ as $t(n) = s(n_0) * \ldots *s(n_{|n|})$. Then the canonical map $\N \to A$ will be tracked by $t$.
\end{proof}

To summarise:
\begin{theo}{ModHerbrandtopos}
The category $\HAsm[\pca{P}]$ is a locally cartesian closed regular category with nno and stable colimits.
\end{theo}

\section{Gamma and nabla}

In this section we show that some of the theory developed for modified realizability also applies to Herbrand realizability. In particular, we show that the facts established on pages 281 and 282 of Van Oosten's paper on the modified realizability topos \cite{vanoosten97} hold for the Herbrand topos as well. (Warning: We follow the notation of that paper, rather than that of \cite{vanoosten08}.)

First of all, note that $\HAsm[\pca{P}]$ is a full subcategory of $\HT[\pca{P}]$, because every triple $(A, \pca{A}, \alpha)$ can be considered as an object $(A, =)$ of $\HT[\pca{P}]$ as follows:
\begin{eqnarray*}
\llbracket a = a' \rrbracket & = & \left\{ \begin{array}{cl}
(\alpha(a), \pca{A}) & \mbox{if } a = a' \\
(\emptyset, \pca{A}) & \mbox{otherwise}
\end{array} \right.
\end{eqnarray*}
In addition, we have a functor $\nabla: \Sets \to \HAsm[\pca{P}]$ which sends a set $X$ to the Herbrand assembly $(X, \pca{P}, \phi)$ with $\phi(x) = !\pca{P}$ for all $x \in X$. Taking the composition of these two functors we obtain a functor $\Sets \to \HT[\pca{P}]$ which we will also denote by $\nabla$.

\begin{rema}{warningonnabla} This is \emph{not} the constant objects functor as defined in \cite{vanoosten08} (and denoted $\nabla$ there).
\end{rema}

\begin{prop}{existence}
The functor $\nabla$ has a finite limit preserving left adjoint \[ \Gamma: \HT[\pca{P}] \to \Sets. \] Moreover, $\Gamma\nabla \cong 1$.
\end{prop}
\begin{proof}
As usual, we send an object $(X, =)$ to $X_0/\sim$, where $X_0 = \{ x \in X \, : \, \llbracket x = x \rrbracket_0 \mbox{ inhabited} \}$ with $x \sim x'$ if $\llbracket x = x' \rrbracket_0$ inhabited. In addition, the transpose of a function $f: \Gamma(X, =) \to Y$ is the function $(X, =) \to \nabla Y$ represented by:
\begin{eqnarray*}
F(x, y) & = & \left\{ \begin{array}{cl}
\llbracket x = x \rrbracket & \mbox{if } x \in X_0 \mbox{ and } f([x]) = y, \\
(\emptyset, \llbracket x = x \rrbracket_1) & \mbox{otherwise.} \end{array} \right.
\end{eqnarray*}
Note that therefore the unit $\eta: (X, =) \to \nabla \Gamma (X, =)$ is represented by $H: X \times X_0 \to \Sigma$ with $H(x, [x']) = \llbracket x = x \rrbracket$ if $x \in [x']$, and $(\emptyset, \llbracket x = x \rrbracket_1)$ otherwise.
\end{proof}

\begin{lemm}{doublenegation} 
Let $(A_0, A_1) \in \Sigma$. Then $A_0$ is inhabited iff $(\lnot\lnot(A_0, A_1))_0$ is inhabited; in which case, $<\lambda p. <>> \in (\lnot\lnot(A_0, A_1))_0$.
\end{lemm}
\begin{proof}
This follows from the fact that $\lnot(A_0, A_1) = (C_0, C_1)$, where $C_1$ is the set of codes of functions which map elements from $!A_1$ to the empty sequence, and $C_0 = !C_1 - \{ <> \}$ if $A_0$ is empty, and $C_0 = \emptyset$ otherwise. Hence $\lambda p.<> \in C_1$ always and $<\lambda p.<>> \in C_0$ if $A_0$ is empty.
\end{proof}

\begin{prop}{charHerbrandassemblies}
For an object $(X, =)$ in $\HT[\pca{P}]$ the following are equivalent:
\begin{enumerate}
\item $\eta_{(X, =)}$ is a monomorphism.
\item $(X, =)$ is $\lnot\lnot$-separated.
\item $(X, =)$ is isomorphic to a Herbrand assembly.
\end{enumerate}
\end{prop}
\begin{proof}
$1 \Rightarrow 2$: Suppose $\eta$ is mono; so
\[ H(x, [z]) \land H(x', [z]) \to x = x' \]
holds. Suppose $<a_1, \ldots, a_n>$ is an actual realizer for this. Furthermore, suppose $b \in \llbracket x = x \rrbracket_0, c \in \llbracket x' = x' \rrbracket_0$ and $d \in \lnot\lnot \llbracket x = x' \rrbracket_0$. Then $b \in H(x, [x])_0, c \in H(x', [x'])_0$ and $[x] = [x']$ by the previous lemma, so 
\[ a_1({\bf p}bc) *  \ldots * a_n({\bf p}bc) \in \llbracket x = x' \rrbracket_0 ; \]
similar for potential realizers. So
\[ x = x \land x' = x' \land \lnot\lnot(x = x') \to x = x' \]
holds and $(X, =)$ is $\lnot\lnot$-separated.

$2 \Rightarrow 3$: Suppose $t$ is an actual realizer of $x = x \land x' = x' \land \lnot\lnot(x = x') \to x = x'$. Define $(A, \pca{A}, \alpha)$ by
\begin{eqnarray*}
A & = & X_0, \\
\pca{A} & = & \bigcup_{x, x' \in X} !\llbracket x = x' \rrbracket_1, \\
\alpha(a) & = & \uparrow_{!\pca{A}} \{ < s > \, : \, (\exists x, x' \in a) \, s \in \llbracket x= x' \rrbracket_0 \}.
\end{eqnarray*}
One easily sees that there is a map $F: (X, =) \to (A, \pca{A}, \alpha)$ defined by 
\begin{eqnarray*}
F(x, a) & = & \left\{ \begin{array}{cl}
\llbracket x = x \rrbracket \land (\alpha(a), \pca{A}) & \mbox{if } x \in a, \\
\llbracket x = x \rrbracket \land (\emptyset, \pca{A}) & \mbox{otherwise.} \end{array} \right.
\end{eqnarray*}
To see that $F(x, a) \land F(x', a) \to x = x'$ is realized (i.e., that $F$ is monic), one uses that if there are actual for realizers $F(x, a)$ of $F(x',a)$, then $\llbracket x = x' \rrbracket_0$ must be inhabited; but then one can compute an element in this set using $t$ and the previous lemma.

To show that $a = a \to (\exists x \in X) \, F(x, a)$ is realized (i.e., to show that $F$ is epic), one needs an element $s \in \pca{P}$ which uniformly in $a \in A$ sends elements from $!\pca{A}$ to elements in
\[ !\bigcup_{x \in X} !F(x, a), \]
in such a way that if $m \in \alpha(a)$, then
\[s \cdot m \in \uparrow \{ <n> \, : \, (\exists x \in X) \, n \in F(x, a)_0 \}. \]
This can be done: for if $k = <k_1, \ldots, k_n> \in !\pca{A}$, then each $k_i$ belongs to some $!\llbracket x = x' \rrbracket_1$. Since this equality is symmetric and transitive, we can compute from $k_i$ an element $tk_i \in !\llbracket x = x \rrbracket_1$. Then ${\bf p}(tk_i, k) \in !F(x, a)_1$ and hence $<{\bf p}(tk_i, k)>_i \in ! \bigcup_{x \in X} !F(x, a)$. If $<k_1, \ldots, k_n>$ also belongs to $\alpha(a)$, then some $k_i$ belongs to $\llbracket x = x' \rrbracket_0$ with $x, x' \in a$. Then ${\bf p}(tk_i, k) \in F(x, a)_0$ and hence $s \cdot m \in  \uparrow \{ <n> \, : \, (\exists x \in X) \, n \in F(x, a)_0 \}$, as desired.

The implication $3 \Rightarrow 1$ is left to the reader.
\end{proof}

\begin{prop}{charnablas}
For an object $(X, =)$ in $\HT[\pca{P}]$ the following are equivalent:
\begin{enumerate}
\item $\eta_{(X, =)}$ is an isomorphism.
\item $(X, =)$ is isomorphic to an object of the form $\nabla Z$.
\item $(X, =)$ is a $\lnot\lnot$-sheaf.
\end{enumerate}
\end{prop}
\begin{proof} $1 \Rightarrow 2$ is trivial.

$2 \Rightarrow 3$: One easily checks by hand that if $f: (X, =) \to (Y, =)$ is a dense mono and $g: (X, =) \to \nabla Z$ is any map, then the relation $g \circ f^{-1}$ is actually a function.

$3 \Rightarrow 1$: If $(X, =)$ is a $\lnot\lnot$-sheaf, then it is certainly separated. Hence $\eta_{(X,=)}$ is monic, by the previous proposition; as it is also dense, it follows that it has a left inverse. But then it is not hard to see that it must be an isomorphism.
\end{proof}

We take a closer look at the objects in the image of $\nabla$. 

\begin{lemm}{useofzero}
If $(A, \pca{A}, \alpha)$ is a Herbrand assembly and there is an element $e \in \pca{P}$ such that $e \in \pca{A}$ and $<e> \in \alpha(a)$ for all $a \in A$, then $(A, \pca{A}, \alpha) \cong \nabla A$.
\end{lemm}
\begin{proof}
Obvious.
\end{proof}

\begin{theo}{nablaandlogic} The functor $\nabla: \Sets \to \HT[\pca{P}]$  preserves and reflects the structure of a locally cartesian closed pretopos. In particular, it preserves and reflects validity of first-order formulas.
\end{theo}
\begin{proof}
One easily checks by hand that $\nabla$ preserves and reflects quotients of equivalence relations. Therefore it suffices to prove that the functor $\nabla: \Sets \to \HAsm[\pca{P}]$ preserves and reflects the structure of a  locally cartesian closed regular category with sums, because this structure is preserved and reflected by the inclusion $\HAsm[\pca{P}] \to \HT[\pca{P}]$. But then the result follows immediately from the constructions we gave in Section 4 and the previous lemma. 
\end{proof}

We have just seen that $\nabla: \Sets \to \HT[\pca{P}]$ preserves and reflects first-order logic. But it does not preserve the natural numbers object, because:

\begin{prop}{nnoinHerbrandtopos}
The Herbrand assembly $(\NN, {\cal N}, \nu)$ is the natural numbers object in the Herbrand topos.
\end{prop}
\begin{proof}
Since the successor map $s: (\NN, {\cal N}, \nu) \to (\NN, {\cal N}, \nu)$ is monic in the Herbrand topos, the nno there will be the smallest subobject of $(\NN, {\cal N}, \nu)$ containing 0 and closed under this map (see, for example, \cite[Corollary D5.1.3]{johnstone02b}). Because any subobject of a separated object is automatically separated, this smallest subobject is the object $(\NN, {\cal N}, \nu)$ itself.
\end{proof}

Therefore $\nabla \NN$ is a nonstandard model of arithmetic in the Herbrand topos. It is this model which underlies the realizability interpretation of nonstandard arithmetic defined in Section 5 of \cite{bergbriseidsafarik12}. 

\section{Projectives}

In this section we study the projective objects in the Herbrand assemblies and the Herbrand topos.

\begin{defi}{partherbassembly} We call a Herbrand assembly $(A, {\cal A}, \alpha)$ \emph{partitioned}, if for every $a \in A$ there is an element $g_a \in \pca{P}$ such that
\[ \alpha(a) = \uparrow_{!\cal A} <g_a>. \]
\end{defi}

We first note that:

\begin{lemm}{easyfactsaboutpartassemblies} We have:
\begin{enumerate}
\item The natural numbers object $(\NN, {\cal N}, \nu)$ is partitioned.
\item Every object of the form $\nabla Z$ is isomorphic to a partioned Herbrand assembly and hence every Herbrand assembly is a subobject of a partitioned one.
\item Partitioned Herbrand assemblies are closed under finite limits.
\item Every Herbrand assembly can be covered by a partitioned one.
\item Every retract of a partitioned Herbrand assembly is isomorphic to a partitioned Herbrand assembly.
\end{enumerate}
\end{lemm}
\begin{proof}
Items 1 and 2 are obvious, so we concentrate on the others.

It is clear that the terminal object in the category of Herbrand assemblies is partitioned and that partitioned Herbrand assemblies are closed under equalizers. Moreover, if $(A, {\cal A}, \alpha)$ and $(B, {\cal B}, \beta)$ are partitioned Herbrand assemblies, then the following partitioned Herbrand assembly is a product in the category of Herbrand assemblies:
\begin{eqnarray*}
C & = & A \times B, \\
{\cal C} & = & {\cal A} \otimes {\cal B}, \\
\gamma(a, b) & = & \uparrow_{!({\cal A} \otimes {\cal B})} <{\bf p}g_ag_b>.
\end{eqnarray*}

A Herbrand assembly $(A, {\cal A}, \alpha)$ can be covered by the partioned Herbrand assembly $(A', {\cal A'}, \alpha')$, where
\begin{eqnarray*}
A' & = & \{ (a, n) \in A \times \pca{P} \, : \, a \in A, n \in \alpha(a) \}, \\
{\cal A'} & = & !{\cal A} \mbox{ and } \\
\alpha'(a, n) & = & \uparrow_{!!{\cal A}} <n>,
\end{eqnarray*}
via the projection. The easy details are left to the reader.

If $(A, {\cal A}, \alpha)$ is a retract of the partitioned Herbrand assembly $(B, {\cal B}, \beta)$ via maps $f: A \to B$ and $g: B \to A$ with $gf = \id$, then $(A, {\cal A}, \alpha)$ is isomorphic to the partioned assembly $(A, {\cal B}, \beta f)$.
\end{proof}

\begin{prop}{partHassemblproj}
Partitioned Herbrand assemblies are projective in the category of Herbrand assemblies, both internally and externally.
\end{prop}
\begin{proof}
We first show that partitioned Herbrand assemblies are externally projective. So suppose $p: (B, {\cal B}, \beta) \to (A, {\cal A}, \alpha)$ is a cover between Herbrand assemblies and $(A, {\cal A}, \alpha)$ is partitioned. Then there is an element $r \in \pca{P}$ such that $r$ is defined on all elements of $!{\cal A}$ and then yields values in $!{\cal B}$, and
\[ (\forall a \in A) \, (\forall n \in \alpha(a)) \, (\exists b \in p^{-1}(a)) \, r \cdot n \in \beta(b). \]
In particular,
\[ (\forall a \in A) \, (\exists b \in p^{-1}(n)) \, r \cdot <g_a> \in \beta(b). \]
Using choice in the metatheory, this means that there is a function $f: A \to B$ such that $pf = \id$ and
\[ (\forall a \in A) \, r \cdot <g_a> \in \beta(fa). \]
Therefore the function $f$ is a section of $p$ tracked by
\[ s \cdot <m_1, \ldots, m_k> = r \cdot <m_1> * \ldots * r \cdot <m_k>. \]
Since the terminal object is partitioned (item 3 of the previous lemma) and therefore externally projective, this implies that externally projective objects are also internally projective. 
\end{proof}

\begin{coro}{regexcompl}
Up to isomorphism, the partitioned Herbrand assemblies are the projective objects in the category of Herbrand assemblies. In particular, this category is equivalent to the reg/lex-completion of its full subcategory on the partitioned Herbrand assemblies.
\end{coro}
\begin{proof}
The first statement follows from the previous proposition and items 4 and 5 of the previous lemma; the second follows from the characterisation of reg/lex-completions in 
\cite{carboni95}.
\end{proof}

\begin{coro}{nnoprojectiveinassemblies}
The natural numbers object is both internally and externally projective in the category of Herbrand assemblies.
\end{coro}

In contrast, it is not generally true that the natural numbers object is externally projective in the entire topos. Indeed, this fails if we take as our pca $\pca{P}$ Kleene's first pca $K_1$.  To define this pca we need to fix an enumeration of the partial recursive functions satisfying some properties (for the precise properties that one needs, see \cite[pages ix and 15]{vanoosten08}). Then $K_1$ has as underlying set the natural numbers and the partial application $n \cdot m$ is defined to be the result of applying the $n$th partial recursive function to the natural number $m$.

\begin{prop}{nnonotprojectiveinHerbrandtopos}
If $\pca{P} = K_1$, then the natural numbers object is not externally projective in the Herbrand topos.
\end{prop}
\begin{proof}
Let $A_0$ and $A_1$ be two recursively inseparable subsets of $\NN$ and let $(\NN \times \{ 0, 1 \}, E)$ be the object of the Herbrand topos which has no actual or potential realizers for nontrival equalities, and whose existence predicate $E = \llbracket - = - \rrbracket$ is given by:
\begin{eqnarray*}
E(n, i)_1 & = & \NN \otimes \{ i \}, \\
E(n, i)_0 & = & \left\{ \begin{array}{cl}
\emptyset & \mbox{if } n \in A_{1-i}, \\
\uparrow_{!E(n,i)_1} <{\bf p}ni> & \mbox{otherwise.}
\end{array} \right.
\end{eqnarray*}
Now consider the projection $p: (\NN \times \{0,1\}, E) \to (\NN, {\cal N}, \nu)$ given by
\[ p((n, i), m) = \left\{ \begin{array}{cl} 
E(n,i) \land (\nu(m), {\cal N}) & \mbox{if } n =m, \\
\emptyset & \mbox{otherwise.}
\end{array} \right. \]
This map is surjective, because its surjectivity is realized by the element $s \in \pca{P}$ satisfying
\[ s \cdot <n_1, \ldots, n_k> = <<{\bf p}n_10>, <{\bf p}n_11>, \ldots, <{\bf p}n_k0>, <{\bf p}n_k1>>. \]
But, on the other hand, this map does not have a section, because a realizer $r$ for such a map would allow one to recursively separate the sets $A_0$ and $A_1$ (for $n \in \NN$ compute the second projection of $(r \cdot <n>)_1$; this always yields either 0 or 1 and for $n \in A_i$ it yields $i$).
\end{proof}

\begin{coro}{noexreg}
If $\pca{P} = K_1$, then not every object in the Herbrand topos can be covered by a Herbrand assemby. In particular, the Herbrand topos is not the ex/reg-completion of the category of Herbrand assemblies.
\end{coro}

\section{Logical features of the Herbrand topos}

In this section we investigate the validity in the Herbrand realizability of some significant logical principles. Such questions are pca dependent and here we restrict attention to the Herbrand topos based on the pca $K_1$. More information on the principles we consider can be found in \cite[Chapter 4]{troelstravandalen88a}.

\begin{lemm}{herbrandcomputable} A function $f: \NN \to \NN$ is tracked (as a morphism $(\NN, {\cal N}, \nu) \to (\NN, {\cal N}, \nu)$ in the Herbrand topos) iff it is bounded by a computable function. Indeed, from a tracking one can compute a bound and vice versa.
\end{lemm}
\begin{proof}
Note that it is necessary and sufficient for a function $f: (\NN, {\cal N}, \nu) \to (\NN, {\cal N}, \nu)$ to be tracked that one can compute for every $n$ a sequence $r(n) = <n_1, \ldots, n_k>$ such that $f(n) = n_i$ for some $i \leq k$. Since one can compute maxima, this implies that $f$ is bounded by a computable function $g$. If, on the other hand, $g$ is a computable function bounding $f$, then $r(n) = <0, \ldots, g(n)>$ shows that $f$ is tracked.
\end{proof}

\begin{coro}{boundedCT}
In the Herbrand topos the following bounded form of Church's Thesis holds:
\[ (\forall x \in \NN) \, (\exists y \in \NN) \, \varphi(x, y) \to (\exists e \in \NN) \, (\forall x \in \NN) \, \big( e \cdot x \downarrow \land \, (\exists y \leq e \cdot x) \, \varphi(x, y) \big). \]
\end{coro}
\begin{proof}
Follows from the previous lemma and \refcoro{nnoprojectiveinassemblies}.
\end{proof}

\begin{theo}{someprinciples}
In the Herbrand topos, the weak law of excluded middle
\[ \lnot \varphi \lor \lnot \lnot \varphi \]
is valid. Hence the De Morgan laws hold, as does:
\begin{equation} \label{formofLEM} (\forall x \in \NN) \, ( \, P(x) \lor \lnot P(x) \,) \to (\forall x \in \NN) \, P(x) \lor \lnot(\forall x \in \NN) \, P(x). 
\end{equation}
In particular, the law of excluded middle is valid for $\Pi_1^0$-formulas. But Markov's principle fails and so does Church's thesis.
\end{theo}
\begin{proof}
It follows from the proof of \reflemm{doublenegation} that $<\lambda p. <>>$ is an actual realizer of $\lnot \varphi$ or of $\lnot \lnot \varphi$ (depending on whether $\varphi$ has an actual realizer or not). Hence ${\bf p}(<\lambda p. <>>,<\lambda p. <>>)$ is an actual realizer of $\lnot \varphi \lor \lnot \lnot \varphi$. The De Morgan laws (in particular, $\lnot (\varphi \land \psi) \to \lnot \varphi \lor \lnot \psi$), the principle in (\ref{formofLEM}) and the law of excluded middle for $\Pi^0_1$-formulas are immediate consequences of this.

Nevertheless, not all classical arithmetic is valid in the Herbrand topos: this follows from the previous corollary. Therefore Markov's Principle must fail (because Markov's Principle together with (\ref{formofLEM}) implies full classical arithmetic). Alternatively, one can argue directly for the failure of Markov's Principle by showing that a realizer for it would allow one to solve the halting problem (along the same lines as for modified realizability; see \cite[Proposition 5.9]{bergbriseidsafarik12}).

Church's Thesis in the form ${\rm CT}^{\lor}_0$ is incompatible with the law of excluded middle for $\Pi_1^0$-formulas (again, use two recursively inseparable r.e.~subsets of $\NN$, or see \cite[Section 4.3]{troelstravandalen88a}). Also ${\rm CT}$ fails: in fact, it follows from \reflemm{herbrandcomputable} that every bounded function $f: \NN \to \NN$ is tracked.
\end{proof}

\begin{coro}{continuity}
Continuity principles, like every function $f: \R \to \R$ is continuous, fail in the Herbrand topos.
\end{coro}
\begin{proof}
Because these are incompatible with the law of excluded middle for $\Pi^0_1$-formulas (see, for example, \cite[Proposition 4.6.4]{troelstravandalen88a}).
\end{proof}

\begin{coro}{notordinaryrealizability}
The Herbrand topos is not equivalent to a realizability topos over an (order-)pca.
\end{coro}
\begin{proof}
Because in such toposes Markov's Principle holds.
\end{proof}

\begin{coro}{twovaluednotboolean}
The Herbrand topos is two-valued, but not boolean.
\end{coro}
\begin{proof}
The terminal object has only two subobjects in the category of Herbrand assemblies; since the terminal object is separated and separated objects are closed under subobjects, the same applies to the Herbrand topos. Hence the Herbrand topos is two-valued. Nevertheless, it is not boolean, since Markov's Principle fails.
\end{proof}

\begin{defi}{tree} (See \cite[page 186]{troelstravandalen88a}.) In the following we will mean by a \emph{tree} an inhabited and decidable set of finite sequences of natural numbers, closed under predecessors. A tree will be called \emph{finitely branching} if
\[ (\forall n \in T) \, (\exists z \in \NN) \, (\forall x\in \NN) \, ( \, n * <x> \in T \to x \leq z). \]
Finally, by an \emph{infinite path} in a tree $T$ we will mean a function $\alpha: \NN \to \NN$ such that $\overline{\alpha}n = < \alpha0, \ldots, \alpha(n-1)> \in T$ for all $n \in \NN$.
\end{defi}

\begin{lemm}{fansuniform}
The collection of infinite paths in a finitely branching tree is a uniform object. More precisely, if $T$ is a finitely branching tree, then one can compute from a realizer for the statement that $T$ is a finitely branching tree an element $n \in \pca{P}$ such that $<n>$ is a common realizer for all infinite paths in the tree.
\end{lemm}
\begin{proof}
From a realizer for the statement that $T$ is finitely branching tree one can compute a bound $f(n)$ on the value of $\alpha(n)$ for every infinite path $\alpha$. This is sufficient in view of  \reflemm{herbrandcomputable}.
\end{proof}

\begin{prop}{koenig}
In the Herbrand topos K\"onig's Lemma holds.
\end{prop}
\begin{proof} Recall that K\"onig's Lemma says that every finitely branching tree which is infinite contains an infinite path. Of course, K\"onig's Lemma is true, but the question is whether we can compute a realizer for an infinite path. The previous lemma, however, guarantees that we can.
\end{proof}

\begin{prop}{fan}
In the Herbrand topos the Fan Theorem holds.
\end{prop}
\begin{proof}
Suppose we have realizers for the statements that $T$ is a finitely branching tree, that $A(x)$ is a property of its nodes, inherited by successors, and that for every infinite path $\alpha$ there is a natural number $n \in \NN$ such that $A(\overline{\alpha}n)$ holds. From a realizer for the first statement we compute a common realizer for all infinite paths. Then from this and a realizer for the third statement we compute a finite sequence $<n_1, \ldots, n_k>$ such that for each infinite path there is an $i$ such that $A(\overline{\alpha}n_i)$ holds. But as $A$ is inherited by successors, we must have $A(\overline{\alpha}n)$ for all infinite paths $\alpha$, if $n = {\rm max}(n_1, \ldots, n_k)$.
\end{proof}

\section{Conclusion}

We have introduced a new topos and established some of its basic properties. Since this paper was written, a number of connections to other toposes have been uncovered: Jaap van Oosten has shown that the Herbrand topos is a  subtopos of the corresponding modified realizability topos, while Peter Johnstone has shown that it is the Gleason cover of the corresponding realizability topos \cite{johnstone13}. In addition, we have shown \cite{berg13} that there is a topos for the nonstandard functional interpretation developed in \cite{bergbriseidsafarik12} which is related to the modified Diller-Nahm topos (see \cite{streicher06} and \cite{biering08}) in the same way as the Herbrand topos is related to the modified realizability topos. But we expect that much more can be said.

\bibliographystyle{plain} \bibliography{ast}

\end{document}